\documentclass{amsart}

\usepackage{color,graphicx,amssymb,latexsym,amsfonts,txfonts,amsmath,amsthm}
\usepackage{pdfsync}
\usepackage{amsmath,amscd}
\usepackage[all,cmtip]{xy}

\usepackage{hyperref}
\hypersetup{
    colorlinks=true,       
    linkcolor=blue,          
    citecolor=blue,        
    filecolor=blue,      
    urlcolor=blue           
}

\newtheorem{theo}{Theorem}
\newtheorem{coro}{Corollary}

\newtheorem{lemm}{Lemma}

\theoremstyle{remark}
\newtheorem{rema}{Remark}

\begin{document}

\title{Constructing Jacobian varieties with many elliptic curves}
\author{Ruben A. Hidalgo}
\address{Departamento de Matem\'atica y Estad\'{\i}stica, Universidad de La Frontera. Temuco, Chile}
\email{ruben.hidalgo@ufrontera.cl}
\thanks{Partially supported by Project Fondecyt 1190001}


\begin{abstract}
Given two elliptic curves, $E_{1}$ and $E_{2}$, Earle provided an explicit genus two Riemann surface $R_{2}$ such that $JR_{2} \cong_{isog.} E_{1} \times E_{2}$.
In this paper, given $s \geq 3$ elliptic curves $E_{1},\ldots, E_{s}$,
we construct an explicit closed Riemann surface $R_{s}$, of genus $g=1+2^{s-2}(s-2)$, such that  $JR_{s} \cong_{isog.} E_{1} \times \cdots \times E_{s} \times A$, where $A$ is also a product of at least $s(s-3)/2$ elliptic curves and jacobian varieties of some hyperelliptic Riemann surfaces, all of these curves explicitly given in terms of the given elliptic curves. In particular, for every triple $E_{1}, E_{2}, E_{3}$ of elliptic curves this provides an explicit Riemann surface $R_{3}$ of genus three with  $JR_{3} \cong_{isog.} E_{1} \times E_{2} \times E_{3}$.

\end{abstract}

\maketitle

\section{Introduction}
Let $R$ be a closed Riemann surface of genus $g \geq 2$ and $JR$ be its jacobian variety. As $JR$ is a particular type of principally polarized abelian variety (p.p.a.v.), it follows from the Poincar\'e reducibility theorem, $JR\cong_{isog.}A_{1}^{n_{1}} \times \cdots \times A_{s}^{n_{s}}$ (the symbol $\cong_{isog.}$ stands for {\it isogenous}), where each $A_{j}$ is a simple abelian variety and each $n_{j}$ is a positive integer (the factors $A_{j}$ and the integers $n_{j}$ are uniquely determined up to isogeny and permutation). If $g \leq 4$, then it is possible to assume all these factors $A_{j}$ to be (isogenous to) jacobian varieties. For $g \geq 5$, this is in general not possible (in fact, it might be that $JR$ is a non-simple p.p.a.v., but not isogenous to a product of at least two jacobian varieties).

In recent years there has been an interest in constructing examples of closed Riemann surfaces whose jacobian varieties are isogenous to a product with many elliptic factors (see, for instance,  \cite{Earle, Earle2, HN, Kani, Nakajima, Magaard, Paulhus1, Paulhus2, PR, RR, CRR, Yamauchi}). Such an interest come from different contexts, for instance, in cryptography \cite{BL, DK} and in arithmetics and number theory \cite{Cardona, FK, CRR}.  For many values of $g \leq 1297$, 
in \cite{E-S}, Ekedahl and Serre constructed examples of closed Riemann surfaces of genus $g$ whose jacobian varieties are  isogenous to a product of $g$ elliptic curves (completely reduced). In \cite{PR} there has been constructed more such kind of examples for some of the missing genera.  It is still not known if such kind of examples exist in every genus or even if they exist for infinitely many values of $g$. In this paper we are interested in the following (which is somehow related to the above) questions.


{\bf Q:} {\it 
(1) Let $E_{1},\ldots, E_{r}$, where $r \geq 2$, be elliptic curves.  Find explicit equations (in terms of these elliptic curves) of a closed Riemann surface $R$ of minimal genus $e(E_{1},\ldots,E_{r})$ with $JR \cong_{isog.}E_{1} \times \cdots \times  E_{r} \times A$, where $A$ is also a product of jacobian varieties. (2) Determine the maximum $e(r)$ of all these minimal genera obtained by varying all the $r$ elliptic curves. }


Observe that $e(E_{1},\ldots,E_{r}) \geq r$ and equality works for Ekedahl-Serre's examples. We may think of $e(E_{1},\ldots,E_{r})-r$ as the defect for $Z=E_{1} \times \cdots \times E_{r}\in {\mathcal H}_{r}$ (where  ${\mathcal H}_{r}$ is the upper-half Siegel space) to be isogenous to a jacobian variety.  The 
${\rm Sp}(2r,{\mathbb Q})$-orbit of $Z$ is dense and any point in this orbit is isogenous to $Z$.
As the dimension of ${\mathcal H}_{r}$ is $r(r+1)/2$ and that of its 
jacobian sublocus is $3(r-1)$,  it follows that, for $r \in \{2,3\}$, there is a dense set of jacobians being isogenous to $Z$. This, in particular, asserts that $e(E_{1},E_{2})=2$ and $e(E_{1},E_{2},E_{3})=3$. (For every $r$ this also asserts that the jacobian variety of every surface of genus $r$ is a limit point of the ${\rm Sp}(2r,{\mathbb Q})$-orbit of $Z$, in particular, that ${\rm Sp}(2r,{\mathbb Q})$-translates of ${\mathcal H}_{1}^{r}$ provides a dense subset.) 
Given any two elliptic curves $E_{1}$, $E_{2}$, Earle \cite{Earle,Earle2} (see also  \cite{Gaudry, HN, Hermite,Magaard}) described an explicit equation of a genus two  Riemann surface $R_{2}$, in terms of the two given elliptic curves, such that $JR_{2}\cong_{isog.}E_{1} \times E_{2}$ (we recall this in Theorem \ref{r=2} for completeness). In this paper, given three elliptic curves $E_{1}, E_{2}, E_{3}$, we describe explicitly a genus three closed Riemann surface $R_{3}$, in terms of the given elliptic curves, such that $JR_{3}\cong_{isog.}E_{1} \times E_{2} \times E_{3}$ (Theorem \ref{coro1}).  

In Section \ref{Sec:preliminaries} we recall a method due to Kani and Rosen \cite{K-R} for finding isogenous decompositions of the jacobian varieties of closed Riemann surfaces admitting non-trivial groups of automorphisms.  In Section \ref{Sec:fiberproducts} we apply such a method to two explicit constructions. In the first construction, 
given  $s \geq 3$ elliptic curves $E_{1}, \ldots E_{s}$, we obtain an explicit 
closed Riemann surface ${R}_{s}$, of genus $g=1+2^{s-2}(s-2)$, such that $JR_{s}\cong_{isog.}E_{1}\times \cdots \times E_{s} \times A$, where 
$A$ is a product of at least $s(s-3)/2$ elliptic curves and some jacobian varieties of hyperelliptic curves, each of them explicitly in terms of the given $s$ elliptic curves (Theorem \ref{construccion}). This permits to observe that: (i) $e(r) \leq 1+2^{(r-2)/2}r$, for $r \geq 4$ even,  and (ii) $e(r) \leq 1+2^{(r-3)/2}(r-1)$, for $r \geq 5$ odd (Theorem \ref{coro2}). We conjecture these inequalities to be sharp, but unfortunately we do not have a reasonable reason for it.
Our constructions also allows us to produce (i) an explicit  $2$-dimensional family of Riemann surfaces of genus nine whose jacobian varieties are isogenous to the product of nine elliptic curves (Example \ref{coro3}),  (ii) a $2$-dimensional family of Riemann surfaces of genus five and (iii) a $1$-dimensional family of Riemann surfaces of genus thirteen, whose jacobian varieties are isogenous to the product of only elliptic curves (Examples \ref{g=5} and \ref{g=13}).

\section{Some preliminaries}\label{Sec:preliminaries}
\subsection{Some notation and remarks}
Let $\Delta_{1}={\mathbb C}-\{0,1\}$ and, for $s \geq 2$, we set $\Delta_{s}=\{(\lambda_{1},\ldots,\lambda_{s}) \in {\mathbb C}^{s}: \lambda_{j} \in \Delta_{1};\; \lambda_{i} \neq \lambda_{j},\; i \neq j\}$.  
If $\lambda \in \Delta_{1}$, then we set the elliptic curve
$E_{\lambda}: y^{2}=x(x-1)(x-\lambda)$. It is known that
$E_{\lambda}$ and $E_{\mu}$ are isomorphic if and only if there is some $T \in 
{\mathbb G}=\langle u(\lambda)=1/\lambda, V(\lambda)=1-\lambda\rangle \cong {\mathfrak S}_{3}$ so that $\mu=T(\lambda)$.
It can be seen that given two elliptic curves $E_{1}$ and $E_{2}$, there is some $(\lambda_{1},\lambda_{2}) \in \Delta_{2}$ such that $E_{j}$ and $E_{\lambda_{j}}$ are isomorphic for $j=1,2$. 
As for every $\lambda \in \Delta_{1}$ there exist infinitely many values $\mu \in \Delta_{1}$ (in fact a dense subset) so that $E_{\mu}$ and $E_{\lambda}$ are isogenous, it follows  that given $s \geq 3$ elliptic curves $E_{1},\ldots, E_{s}$, there are infinitely many tuples $(\lambda_{1},\ldots,\lambda_{s}) \in \Delta_{s}$ so that $E_{\lambda_{j}}$ and $E_{j}$ are isogenous for each $j=1,\ldots,s$.

\subsection{The jacobian variety of a closed Riemann surface}
A {\it polarized abelian variety} of dimension $g \geq 1$ is a pair $A=(T,Q)$, where $T={\mathbb C}^{g}/L$ is a complex torus of dimension $g$ and $Q$ (called a {\it polarization} of $A$) is a positive-definite Hermitian product in ${\mathbb C}^{g}$ such that ${\rm Im}(Q)$ has integral values over elements of the lattice $L$. There is basis of $L$ for which ${\rm Im}(Q)$ can be represented by the matrix
$$\left( \begin{array}{cc}
0 & D\\
-D & 0
\end{array}
\right)
$$
where $D$ is a diagonal matrix with diagonal entries given by non-negative integers $d_{1},\ldots,d_{g}$, where $d_{j}$ divides $d_{j+1}$. The tuple $(d_{1},\ldots,d_{g})$ is called the {\it polarization type}. When $d_{1}=\cdots=d_{g}=1$, the polarization is called {\it principal} and that the abelian variety is {\it principally polarized}. 
(Not every torus $T$ can be given the structure of a polarized abelian variety; a necessary and sufficient condition is the existence of a holomorphic embedding of $T$ into a complex projective space.) 
Two tori $T_{1}={\mathbb C}^{g}/L_{1}$ and $T_{2}={\mathbb C}^{g}/L_{2}$ are called {\it isogenous} ($T_{1} \cong_{isog.} T_{2}$) if there exits a 
non-constant surjective morphism (i.e., holomorphic and also a group homomorphism) $h:T_{1} \to T_{2}$  (so it has a finite kernel); $h$ is called an {\it isogeny}. 
 In particular, we may talk of {\it isogenous abelian varieties} (the polarization plays no role in this definition).
Let $R$ be a closed Riemann surface of genus $g \geq 1$. Its first homology group $H_{1}(R,{\mathbb Z})$ is isomorphic to ${\mathbb Z}^{2g}$ and its complex vector space $H^{1,0}(R)$ of its holomorphic $1$-forms is isomorphic to ${\mathbb C}^{g}$. There is a natural injective map
$\iota:H_{1}(R,{\mathbb Z}) \hookrightarrow \left( H^{1,0}(R) \right)^{*}: \alpha \mapsto \int_{\alpha} \cdot \;,$
where $\left( H^{1,0}(R) \right)^{*}$ denotes the dual space of $H^{1,0}(R)$.
The image $\iota(H_{1}(R,{\mathbb Z}))$ is a lattice in $\left( H^{1,0}(R) \right)^{*}$ and the quotient $g$-dimensional torus 
$JR=\left( H^{1,0}(R) \right)^{*}/\iota(H_{1}(R,{\mathbb Z}))$
is called the {\it jacobian variety} of $R$.  The intersection product in $H_{1}(R,{\mathbb Z})$ induces a principal polarization on $JR$; that is, $JR$ is a principally polarized abelian variety. 

\subsection{Kani-Rosen's decomposition result}
An abelian variety is called {\it simple} if it is not isogenous to a product of at least two abelian varieties of smaller dimensions. Poincar\'e irreducibility theorem asserts that for every abelian variety $A$ there 
exist simple polarized abelian varieties $A_{1}, \ldots, A_{s}$ and positive integers $n_{1},\ldots,n_{s}$ such that $A\cong_{isog.}A_{1}^{n_{1}} \times \cdots A_{s}^{n_{s}}$. Moreover, the factors $A_{j}$ and the integers $n_{j}$ are unique up to isogeny and permutation of the factors.
In general, to describe these simple factors seems to be a very difficult problem. When the abelian variety $A$ admits a non-trivial group $G$ of automorphisms, then there is a method to compute factors (non-necessarilly simple ones) by using the rational representations of $G$ (the group algebra decomposition) \cite{CR, LR, Anita}. In the particular case that $A=JR$, for $R$ a closed Riemann surface, the following decomposition, in terms of automorphisms of $R$, is due to Kani and Rosen.

\begin{theo}[Kani-Rosen's decomposition \cite{K-R}]\label{coroKR}
Let $R$ be a closed Riemann surface of genus $g \geq 1$ and let $H_{1},\ldots,H_{s}<{\rm Aut}(R)$ such that:
(1) $H_{i} H_{j}=H_{j} H_{i}$, for all $i,j =1,\ldots,s$;
(2) $g_{H_{i}H_{j}}=0$, for $1 \leq i < j \leq s$; and 
(3) $g=\sum_{j=1}^{s} g_{H_{j}}$.
Then 
$JR \cong_{isog.}  JR_{H_{1}} \times \cdots \times JR_{H_{s}}.$
\end{theo}

\subsection{Generalized Humbert curves}\label{Sec:Humbert}
The constructions of the corresponding Riemann surfaces of Theorems \ref{construccion} and \ref{construccion2} will be suitable quotients of certain types of Riemann surfaces called generalized Humbert curves. We proceed to recall them and some of their properties (details can be found in \cite{C-G-H-R}).

Let $S$ be a closed Riemann surface. We say that $S$ is a 
{\it generalized Humbert curve of type $n \geq 2$} if it admits a group $H \cong {\mathbb Z}_{2}^{n}$ of conformal automorphisms such that $S/H$ has genus zero (in particular, it has exactly $n+1$ cone points, each one of order two). The group $H$ is 
called a {\it generalized Humbert group of type $n$}. By the Riemann-Hurwitz formula, the genus of $S$ is $g_{n}=1+ 2^{n-2}(n-3)$. If $n \geq 4$, then: (i) $g_{n} \geq 5$, (ii) $S$ is non-hyperelliptic \cite{C-G-H-R} and (iii) it has a unique generalized Humbert group of type $n$ \cite{H-K-L-P}. 
Let $\pi:S \to \widehat{\mathbb C}$ be a branched regular covering whose deck group is $H$. By post-composing it by a suitable M\"obius transformation, we may assume its branch values to be given by $\infty, 0,1, \lambda_{1}, \ldots, \lambda_{n-2}$, where $\lambda_i \in {\mathbb C}-\{0,1\}$ are pairwise different. It can be observed that the following is a non-singular complex projective algebraic curve

$$
C(\lambda_{1},\ldots,\lambda_{n-2}):=\left \{ \begin{array}{rcc}
x_{1}^{2}+x_{2}^{2}+x_{3}^{2}&=&0\\
\lambda_{1}x_{1}^{2}+x_{2}^{2}+x_{4}^{2}&=&0\\
\vdots\hspace{1cm} &\vdots &\vdots\\
\lambda_{n-2}x_{1}^{2}+x_{2}^{2}+x_{n+1}^{2}&=&0\\
\end{array}\right \}\subset {\mathbb P}^{n}.
$$

\begin{rema} \label{remark1}
Set $\lambda_{0}=1$ and, for $j \in \{0,1,\ldots,n-2\}$, let $C_{j}:=\{ \lambda_{j}x_{1}^{2}+x_{2}^{2}+x_{3+j}^{2}=0 \} \subset {\mathbb P}^{2}$ be the classical degree two Fermat curve. The rational map $\pi_{j}:C_{j} \to \widehat{\mathbb C}$, defined by $\pi_{j}([x_{1}:x_{2}:x_{3+j}])=-(x_{2}/x_{1})^{2}$, is a branched regular covering with deck group $H_{j}=\langle a([x_{1}:x_{2}:x_{3+j}])=[-x_{1}:x_{2}:x_{3+j}], b([x_{1}:x_{2}:x_{3+j}])=[x_{1}:-x_{2}:x_{3+j}]\rangle \cong {\mathbb Z}_{2}^{2}$, whose branch values are  $\infty$, $0$ and $\lambda_{j}$. If we consider the fiber product of all these curves, with the given maps, we obtain a
reducible projective algebraic curve with $2^{n-2}$ irreducible components (all of them isomorphic to $C(\lambda_{1},\ldots,\lambda_{n-2})$).
\end{rema}

If we consider the linear automorphisms of ${\mathbb P}^{n}$ given by 
\begin{center}
 $b_{j}([x_{1}:\cdots:x_{n+1}]):=[x_{1}:\cdots:x_{j-1}:-x_{j}:x_{j+1}:\cdots:x_{n+1}]$,
 \end{center}
then $b_{1} b_{2}  \cdots b_{n+1}=1$ and 
$H_{0}=\langle b_{1},\ldots,b_{n}\rangle \cong {\mathbb Z}_{2}^{n}$ is a subgroup of ${\rm Aut}(C(\lambda_{1},\ldots,\lambda_{n-2}))$.
The set of fixed points of  $b_{j}$ in $C(\lambda_{1},\ldots,\lambda_{n-2})$ is given by the intersection 
$${\rm Fix}(b_{j}):=\{x_j=0\} \cap C(\lambda_{1},\ldots,\lambda_{n-2}),$$
which is of cardinality $2^{n-1}$. The only non-trivial elements of $H_{0}$ acting with fixed points on $C(\lambda_{1},\ldots,\lambda_{n-1})$ are the non-trivial powers of the elements $b_{1}, \ldots, b_{n+1}$.

The map 
$\pi_{0}:C(\lambda_{1},\ldots,\lambda_{n-2}) \to \widehat{\mathbb C}: [x_{1}:\cdots:x_{n+1}] \mapsto -\left(x_{2}/x_{1}\right)^{2}$
is a regular branched cover with deck group $H_{0}$ and whose branch values are $\infty, 0,1, \lambda_{1}, \ldots, \lambda_{n-2}$, each one of order two. In other words, $C(\lambda_{1},\ldots,\lambda_{n-2})$ is a generalized Humbert curve of type $n$ and $H_{0}$ is its generalized Fermat group of type $n$. 
It was noticed in \cite{C-G-H-R} that 
$S$ and $C(\lambda_{1},\ldots,\lambda_{n-2})$ are isomorphic Riemann surfaces.

\begin{rema}
Let ${\mathcal F}$ be the collection of all subsets of $\{\infty,0,1,\lambda_{1},\ldots,\lambda_{n-2}\}$ of even cardinality at least four. For each $I \in {\mathcal F}$ we set $R_{I}$ the Riemann surface defined by the algebraic curve $y^{2}=\prod_{\alpha \in I}(x-\alpha)$ (if $\alpha=\infty$, then we delete the factor $(x-\alpha)$). If $|I|=4$, then $R_{I}$ is an elliptic curve and if $|I|=2m \geq 6$, then $R_{I}$ is a hyperelliptic Riemann surface of genus $g=m-1$. In \cite{CHQ} (as an application of Kani-Rosen's result)  it was observed that $JS \cong_{isog.} \prod_{I \in {\mathcal F}} JR_{I}$.
\end{rema}

\subsection{The known genus two situation}
As already said in the introduction, given two elliptic curves $E_{1}$ and $E_{2}$, there is a closed Riemann surface $R_{2}$ of genus two with $JR_{2}$ isogenous to $E_{1} \times E_{2}$. This can be tracked back to \cite{Earle, Earle2, Gaudry, HN, Hermite}. 

\begin{theo}\label{r=2}
Let $E_{1}$ and $E_{2}$ two elliptic curves and let 
$(\lambda_{1},\lambda_{2}) \in \Delta_{2}$ be so that $E_{j}$ is isomorphic to $E_{\lambda_{j}}$, for $j=1,2$. If
$\eta_{1}=(\lambda_{1}-1)/(\lambda_{2}-1)$, $\eta_{2}=\lambda_{2}(\lambda_{1}-1)/\lambda_{1}(\lambda_{2}-1)$, and $R_{2}$ is the genus two Riemann surface defined by the hyperelliptic curve $y^{2}=(x^{2}-1)(x^{2}-\eta_{1})(x^{2}-\eta_{2})$, then $JR_{2}\cong_{isog.}E_{1} \times E_{2}$. 
\end{theo}
\begin{proof}
Consider the genus two curve 
$C: \; y^{2}=(x^{2}-1)(x^{2}-\eta_{1})(x^{2}-\eta_{2})$
and its order two automorphisms 
$a_{1}(x,y)=(-x,y)$ and $a_{2}(x,y)=(-x,-y)$.
Then $H=\langle a_{1},a_{2}\rangle \cong {\mathbb Z}_{2}^{2}$ and $a_{2} a_{1}$ is the hyperelliptic involution. Set $H_{1}=\langle a_{1}\rangle$ and $H_{2}=\langle a_{2}\rangle$. If we consider the elliptic curves 
$F_{1}:\; y^{2}=(x-1)(x-\eta_{1})(x-\eta_{2})$ and $F_{2}:\; y^{2}=x(x-1)(x-\eta_{1})(x-\eta_{2})$,
then $P_{j}:C \to F_{j}$, where $P_{1}(x,y)=x^{2}$ and $P_{2}(x,y)=(x^{2},xy)$, are two-fold branched coverings with respective deck groups $H_{1}$ and $H_{2}$. As $H_{1}H_{2}=H=H_{2}H_{1}$ and $C/H$ is the Riemann sphere (with exactly $5$ cone points, these being $\infty$, $0$, $1$, $\eta_{1}$ and $\eta_{2}$), we may apply Kani-Rosen's result (Theorem \ref{coroKR}) to obtain that $JC \cong_{isog.} F_{1} \times F_{2}$. As, for $L(x)=(1-\eta_{1})(x-\eta_{2})/(1-\eta_{2})(x-\eta_{1})$, one has that $L(1)=1$, $L(\eta_{1})=\infty$, $L(\eta_{2})=0$, $L(\infty)=(1-\eta_{1})/(1-\eta_{2})=\lambda_{1}$ and
$L(0)=\lambda_{2}$, we observe that $E_{1} \cong F_{1}$ and $E_{2} \cong F_{2}$.
\end{proof}

\begin{rema}[The curve $R_{2}$ in terms of fiber products]\label{observa2}
The fiber product of $(E_{\lambda_{1}},\pi_{1})$ and $(E_{\lambda_{2}},\pi_{2})$, where $\pi_{j}(x,y)=x$, (in projective coordinates) is 
$$X_{2}:=\left\{ ([x_{1}:y_{1}:z_{1}],[x_{2}:y_{2}:z_{2}]) \in {\mathbb P}^{2} \times {\mathbb P}^{2}: \; x_{1}/z_{1}=x_{2}/z_{2}, \right.$$
$$\left.
y_{1}^{2}z_{1}=x_{1}(x_{1}-z_{1})(x_{1}-\lambda_{1}z_{1}), \quad y_{2}^{2}z_{2}=x_{2}(x_{2}-z_{2})(x_{2}-\lambda_{2}z_{2}) \right\}.$$

This is an irreducible algebraic curve with exactly $3$ nodes, these being at the points $\infty:=([0:1:0],[0:1:0])$, $p:=([0:0:1],[0:0:1])$ and $q:=([1:0:1],[1:0:1])$. After desingularization, this produces a genus two Riemann surface $R$, and each of these three special points induces two different points ($\infty_{1}$, $\infty_{2}$, $p_{1}$, $p_{2}$, $q_{1}$ and $q_{2}$) on $R$.
On $X_{2}$ we have the order two automorphisms
$$a_{1}([x_{1}:y_{1}:z_{1}],[x_{2}:y_{2}:z_{2}])=([x_{1}:-y_{1}:z_{1}],[x_{2}:y_{2}:z_{2}]),$$ 
$$a_{2}([x_{1}:y_{1}:z_{1}],[x_{2}:y_{2}:z_{2}])=([x:y_{1}:z_{1}]),[x_{2}:-y_{2}:z_{2}]),$$ 
with $H=\langle a_{1}, a_{2} \rangle \cong {\mathbb Z}_{2}^{2}$. The  projection $\pi:X_{2} \to \widehat{\mathbb C}:([x_{1}:y_{1}:z_{1}],[x_{2}:y_{2},z_{2}]) \mapsto x_{1}/z_{1}$, is a Galois branched cover with 
$H$ as its deck group. The branch values of $\pi$ are $\infty$, $0$, $1$, $\lambda_{1}$ and $\lambda_{2}$. 
The Riemann surface $R$ has genus two and admits the group $H$  as group of conformal automorphisms. If $H_{j}=\langle a_{j} \rangle$, it can be seen that the orbifold $R/H_{j}$ has underlying Riemann surface $E_{\lambda_{3-j}}$. The surface $R$ is another model for the genus two Riemann surface $R_{2}$ described in the above theorem.
\end{rema}

\section{The constructions}\label{Sec:fiberproducts}

\subsection{First construction}
Let us assume $s \geq 3$ and consider the set of cardinality $2^{s-1}-1$
$$V_{s}=\{\alpha=(\alpha_{1},\ldots,\alpha_{s}) \in \{0,1\}^{s}-\{(0,\ldots,0)\}: \alpha_{1}+\cdots+\alpha_{s} \; \mbox{is even}\}.$$

If $P:=(\lambda,\mu_{1,1},\mu_{1,2},\mu_{2,1},\mu_{2,2},\ldots,\mu_{s-2,1},\mu_{s-2,2}) \in \Delta_{2s-3}$ and $\alpha=(\alpha_{1},\ldots,\alpha_{s}) \in V_{s}$, then we set 
the following elliptic curves 
$$
E_{1}(P):\; y^{2}=x(x-1)(x-\lambda), \;
E_{2}(P):\; y^{2}=(x-1)(x-\lambda)(x-\mu_{1,1})(x-\mu_{1,2}),
$$
$$
E_{3}(P): \; y^{2}=(x-\mu_{1,1})(x-\mu_{1,2})(x-\mu_{2,1})(x-\mu_{2,2}), \;
E_{4}(P):\; y^{2}=(x-\mu_{2,1})(x-\mu_{2,2})(x-\mu_{3,1})(x-\mu_{3,2}),
$$
$$\vdots$$
$$
E_{s-1}(P):\; y^{2}=(x-\mu_{s-3,1})(x-\mu_{s-3,2})(x-\mu_{s-2,1})(x-\mu_{s-2,2}), \;
E_{s}(P):\; y^{2}=x(x-\mu_{s-2,1})(x-\mu_{s-2,2}),
$$
the complex numbers
$\eta_{0}=-(\mu_{s-2,2})^{-1}$, $\eta_{1}=(1-\mu_{s-2,2})^{-1}$, 
$\eta_{2}=(\lambda-\mu_{s-2,2})^{-1}$, $\eta_{3}=(\mu_{s-2,1}-\mu_{s-2,2})^{-1}$ and, for each $t=1,2$ and $j=1,\ldots,s-3$,  
$\eta_{j,t}=(\mu_{j,t}-\mu_{s-2,2})^{-1}$,
and the following non-zero complex number
$$K_{\alpha}=(-\mu_{s-2,2})^{\alpha_{1}} (\mu_{s-2,2}-1)^{\alpha_{2}}(\mu_{s-2,2}-\lambda)^{\alpha_{2}} (\mu_{s-2,2}-\mu_{s-2,1})^{\alpha_{s}} \prod_{k=1}^{s-3} (\mu_{s-2,2}-\mu_{k,1})^{\alpha_{k+2}}(\mu_{s-2,\mu_{k,2}}-\lambda)^{\alpha_{k+2}}.$$

\begin{theo}\label{construccion}
Let $s \geq 3$ and $P=(\lambda,\mu_{1,1},\mu_{1,2},\mu_{2,1},\mu_{2,2},\ldots,\mu_{s-2,1},\mu_{s-2,2}) \in \Delta_{2s-3}$. 
Let $X \subset {\mathbb C}^{2^{s-1}}$ be the affine curve defined by the following $(2^{s-1}-1)$ elliptic/hyperelliptic equations
$$
\left\{ \begin{array}{c}
w_{\alpha}^{2}=K_{\alpha} z^{\alpha_{1}}(z-\eta_{0})^{\alpha_{1}}
(z-\eta_{1})^{\alpha_{2}}(z-\eta_{2})^{\alpha_{2}}(z-\eta_{3})^{\alpha_{s}}
\prod_{k=1}^{s-3}(z-\eta_{k,1})^{\alpha_{k+2}} (z-\eta_{k,2})^{\alpha_{k+2}},\\ 
\alpha=(\alpha_{1},\ldots,\alpha_{s}) \in V_{s}.\\
\end{array}
\right\}.
$$

Then $X$ defines a closed Riemann surface $R_{s}$ of genus $g=1+2^{s-2}(s-2)$ whose jacobian variety 
is isogenous to the product of the jacobian varieties of the following $\sum_{j=1}^{[s/2]}\binom{s}{2j}$ elliptic/hyperelliptic curves
$$C_{i_{1},\ldots,i_{k}}: \nu^{2}=(\upsilon-\rho_{i_{1},1})(\upsilon-\rho_{i_{1},2})\cdots(\upsilon-\rho_{i_{k},1})(\upsilon-\rho_{i_{k},2}),$$
where $2 \leq k \leq s$ is even, the tuples $(i_{1},\ldots,i_{k})$ satisfy
$1\leq  i_{1} < i_{2} < \cdots < i_{k} \leq s,$
and
$$\rho_{i_{j},1}=\left\{\begin{array}{cl}
\infty, & i_{j}=1\\
1, & i_{j}=2\\
\mu_{r-2,1}, & i_{j}=r \geq 3
\end{array}
\right. \quad 
\rho_{i_{j},2}=\left\{\begin{array}{cl}
0, & i_{j}=1\\
\lambda, & i_{j}=2\\
\mu_{r-2,2}, & i_{j}=r\geq 3
\end{array}
\right.
$$
where, if $\rho_{i_{j},1}=\infty$, then the factor $(u-\rho_{i_{j},1})$ is deleted from the above expression. In particular, $JR_{s}$ contains, in its isogenous factors above, the elliptic curves $E_{1}(P),\ldots, E_{s}(P)$. 
\end{theo}

\begin{rema}\label{fibrado} 
(1) $JR_{s}$ contains at least  $s(s-1)/2$ elliptic curves in its isogenous decomposition. 
(2) Given $s$ elliptic curves, $\widehat{E}^{*}_{1},\ldots,\widehat{E}^{*}_{s}$, it is possible to find a tuple $P \in \Delta_{2s-3}$ such that $\widehat{E}^{*}_{j}$ is isogenous to $E_{j}(P)$, for every $j=1,\ldots, s$. So, if all these elliptic curves $E_{j}(P)$ are isogenous to a given one $E$, then $JR_{s} \cong_{isog.} E^{s} \times A$, where $A$ is isogenous to the product of other elliptic curves and the jacobian varieties of hyperelliptic Riemann surfaces.
\end{rema}

\subsection{Two applicactions}
In Theorem \ref{construccion}, for $s=3$, the surface $R_{3}$ has genus three and $JR_{3} \cong_{isog.} E_{1}(P) \times E_{2}(P) \times E_{3}(P)$. Next, we observe that we may chose the tuple $P \in \Delta_{3}$ such that the three elliptic curves $E_{1}(P), E_{2}(P)$ and $E_{3}(P)$ may be chosen to be any three given elliptic curves.

\begin{theo}\label{coro1}
Let $E_{1}$, $E_{2}$ and $E_{3}$ be three elliptic curves. Choose $(\lambda_{1},\lambda_{2},\lambda_{3}) \in \Delta_{3}$ so that $E_{j}$ is isogenous to $E_{\lambda_{j}}$, for $j=1,2,3$. If $\mu$ is a root of 
$$\lambda_{2}\lambda_{3} \mu^{2}-(\lambda_{1}\lambda_{2}+\lambda_{2}\lambda_{3}+\lambda_{1}\lambda_{3}-\lambda_{1}-\lambda_{3}+1) \mu +\lambda_{1}\lambda_{2}=0,$$
then
$$R_{3}:=\left\{ \begin{array}{lcl}
w_{1}^{2}&=&\mu(\lambda_{3} \mu -1)(\lambda_{3} \mu -\lambda_{1})(\lambda_{3} - 1)
z \left( z-\dfrac{1}{\lambda_{1}-\lambda_{3}\mu} \right) \left(z-\dfrac{1}{\mu (1-\lambda_{3})} \right)\\
w_{2}^{2}&=&-\lambda_{3}\mu^{2}(\lambda_{3}-1)z\left(z+\dfrac{1}{\lambda_{3} \mu}\right)  \left(z-\dfrac{1}{1-\lambda_{3}\mu} \right)\\
w_{3}^{2}&=&-\lambda_{3}\mu^{2}(\lambda_{3}\mu-1)(\lambda_{3}-1)z^{2}\left(z+\dfrac{1}{\lambda_{3} \mu}\right) \left(z-\dfrac{1}{\mu(1-\lambda_{3})}\right) \\
\end{array}
\right\},
$$
defines a closed Riemann surface of genus three with $JR_{3}\cong_{isog.}E_{1} \times E_{2} \times E_{3}$. 
\end{theo}
\begin{proof}
If $\lambda=\lambda_{1}$, $\mu_{1,1}=\mu$ and $\mu_{1,2}=\lambda_{3}\mu$, then $(\lambda,\mu_{1,1},\mu_{1,2}) \in \Delta_{3}$ and 
$$E_{\lambda_{1}}=F_{1}: y^{2}=x(x-1)(x-\lambda_{1}),\;
E_{\lambda_{2}} \cong F_{2}: y^{2}=(x-1)(x-\lambda_{1})(x-\mu)(x-\lambda_{3}\mu),$$
$$E_{\lambda_{3}} \cong F_{3}: y^{2}=x(x-\mu)(x-\lambda_{3}\mu).$$

Theorem \ref{construccion}, applied to the triple $(\lambda,\mu_{1,1},\mu_{1,2})$, asserts that $JR_{3}\cong_{isog.}F_{1} \times F_{2} \times F_{3}$, so isogenous to the product $E_{1} \times E_{1} \times E_{3}$.
In this case the Riemann surface $R_{3}$ is described by the curve 
$$\left\{ \begin{array}{lcl}
w_{1}^{2}&=&\mu(\lambda_{3} \mu -1)(\lambda_{3} \mu -\lambda_{1})(\lambda_{3} - 1)
z \left( z-\dfrac{1}{\lambda_{1}-\lambda_{3}\mu} \right) \left(z-\dfrac{1}{\mu (1-\lambda_{3})} \right)\\
w_{2}^{2}&=&-\lambda_{3}\mu^{2}(\lambda_{3}-1)z\left(z+\dfrac{1}{\lambda_{3} \mu}\right)  \left(z-\dfrac{1}{1-\lambda_{3}\mu} \right)\\
w_{3}^{2}&=&-\lambda_{3}\mu^{2}(\lambda_{3}\mu-1)(\lambda_{3}-1)z^{2}\left(z+\dfrac{1}{\lambda_{3} \mu}\right) \left(z-\dfrac{1}{\mu(1-\lambda_{3})}\right) \\
\end{array}
\right\},
$$
the group $H=\langle a_{1},a_{2}\rangle \cong {\mathbb Z}_{2}^{2}$, where 
$$a_{1}(z,w_{1},w_{2},w_{3})=(z,w_{1},-w_{2},-w_{3}), \; 
a_{2}(z,w_{1},w_{2},w_{3})=(z,-w_{1},w_{2},-w_{3}),$$
each of the three automorphisms $a_{1}$, $a_{2}$ and $a_{3}=a_{1}a_{2}$ acting with exactly four fixed points, 
and the corresponding regular branched cover with $H$ as deck group is 
$$\pi:R_{3} \to \widehat{\mathbb C}: (z,w_{1},w_{2},w_{3}) \mapsto (\lambda_{3} \mu z +1)/z.$$

The non-trivial proper subgroups of $H$ are
$H_{1,2}=\langle a_{1}a_{2}\rangle$, $H_{1,3}=\langle a_{1}a_{3}\rangle=\langle a_{2} \rangle$ and $H_{2,3}=\langle a_{2}a_{3}\rangle=\langle a_{1}\rangle$.
The quotients $R_{3}/\langle a_{1} \rangle$, $R_{3}/\langle a_{2} \rangle$ and $R_{3}/\langle a_{1}a_{2} \rangle$ are of genus one and they correspond, respectively, to the elliptic curves $E_{\lambda_{2}}$, $E_{\lambda_{1}}$ and $E_{\lambda_{3}}$.
\end{proof}

Another consequence of the construction provided by Theorem \ref{construccion} is the following upper bound for $e(r)$ (which we conjecture to be sharps).

\begin{theo}\label{coro2}
$$e(r) \leq \left\{ \begin{array}{ll}
1+2^{(r-2)/2} r, & r \geq 4 \quad \mbox{even}\\
1+2^{(r-3)/2}(r-1), & r \geq 5 \quad \mbox{odd}
\end{array}
\right.
$$
\end{theo}
\begin{proof}
Assume $r \geq 3$ is odd and write $r=2s-3$, where $s \geq 3$. Let us fix $(\lambda_{1},\ldots,\lambda_{r}) \in \Delta_{r}$. Set 
$\lambda=\lambda_{1}$ and, for $j=1,\ldots,s-2$, we set
$\mu_{j,2}=\lambda_{j+1} \mu_{j,1}$ and $\mu_{j,1}$ a root of the polynomial
$$\lambda_{j+1}(1-\lambda_{s-2+j})\mu_{j,1}^{2}+(\lambda_{s-2+j}-\lambda_{j+1}-\lambda_{1}+\lambda_{1}\lambda_{j+1}\lambda_{s-2+j})\mu_{j,1}+(1-\lambda_{1}\lambda_{s-2+j})=0.$$

As a value $\lambda_{j}$ may be changed to some other value $\lambda_{j}'$ (inside an infinite set of values) so that $E_{\lambda_{j}}$ and $E_{\lambda_{j}'}$ are isogenous, we may assume  
$(\lambda,\mu_{1,1},\mu_{1,2},\mu_{2,1},\mu_{2,2},\ldots,\mu_{s-2,1},\mu_{s-2,2}) \in \Delta_{2s-3}$.
If $R_{s}$ is the Riemann surface constructed in Theorem \ref{construccion}, then $JR_{s}$ is isogenous to a product of certain explicit jacobian varieties. It can be seen, from the proof of that theorem, that some of these factors are (isogenous to) the elliptic curves 
$$\begin{array}{lll}
y^{2}&=&x(x-1)(x-\lambda)\\
y^{2}&=&x(x-\mu_{j,1})(x-\mu_{j,2}), \quad j=1,\ldots,s-2,\\
y^{2}&=&(x-1)(x-\lambda)(x-\mu_{j,2}), \quad j=1,\ldots,s-2.
\end{array}
$$

The choice we have made for the tuple $(\lambda,\mu_{1,1},\mu_{1,2},\mu_{2,1},\mu_{2,2},\ldots,\mu_{s-2,1},\mu_{s-2,2}) \in \Delta_{2s-3}$ ensures that they are all isomorphic to the elliptic curves $y^{2}=x(x-1)(x-\rho)$,
where $\rho \in \{\lambda_{1},\ldots,\lambda_{r}\}$.
The case $r \geq 4$ even can be worked similarly, but in this case we add an extra elliptic curve to the $r$ given ones in order to obtain the result as a consequence of the odd situation.
\end{proof}

\subsection{The second construction}
The Riemann surface obtained in the construction given in Theorem \ref{construccion} is one of two (isomorphic) factors of a fiber product. The next  theorem describes a similar construction, but in this case, this surface is an irreducible fiber product.

\begin{theo}\label{construccion2}
If $(\lambda_{1},\ldots,\lambda_{r}) \in \Delta_{r}$, then the affine algebraic curve
$$Y=
\left\{\begin{array}{c}
w_{1}^{2}=(\lambda_{r}-\lambda_{1})u+1, \; 
w_{2}^{2}=(\lambda_{r}-\lambda_{2})u+1, \cdots, 
w_{r-1}^{2}=(\lambda_{r}-\lambda_{r-1})u+1,\\
w_{r}^{2}=-u(\lambda_{r}u+1)((\lambda_{r}-1)u+1)
\end{array}
\right\}
$$
defines a closed Riemann surface $R$ of genus $g=1+2^{r-2}(r-1)$ such that $JR \cong_{isog.} E_{\lambda_{1}} \times \cdots \times E_{\lambda_{r}} \times A_{g-r}$, where $A_{g-r}$ is the product of certain explicit elliptic/hyperelliptic Riemann surfaces.
\end{theo}

\section{Proof of Theorems \ref{construccion} and \ref{construccion2}}
\subsection{Proof of Theorem \ref{construccion}}
Set $E_{j}=E_{j}(P)$, where $j=1,\ldots s$.
Let us consider the (affine) generalized Humbert curve of type $2s-1$
$$
D: \left\{ \begin{array}{cccccc}
z_{1}^{2}+z_{2}^{2}+z_{3}^{2}&=&0,& \lambda z_{1}^{2}+z_{2}^{2}+z_{4}^{2}&=&0,\\
\mu_{1,1}z_{1}^{2}+z_{2}^{2}+z_{5}^{2}&=&0, & \mu_{1,2}z_{1}^{2}+z_{2}^{2}+z_{6}^{2}&=&0,\\
\vdots & \vdots& \vdots & \vdots & \vdots& \vdots\\
\mu_{k,1}z_{1}^{2}+z_{2}^{2}+z_{2k+3}^{2}&=&0,& \mu_{k,2}z_{1}^{2}+z_{2}^{2}+z_{2k+4}^{2}&=&0,\\
\vdots & \vdots& \vdots & \vdots & \vdots& \vdots\\
\mu_{s-2,1}z_{1}^{2}+z_{2}^{2}+z_{2s-1}^{2}&=&0, & \mu_{s-2,2}z_{1}^{2}+z_{2}^{2}+1&=&0.
\end{array}
\right\}.
$$

As observed in Section \ref{Sec:Humbert}, $D$ defines a closed Riemann surface of genus $g_{D}=1+2^{2s-2}(s-2)$, with generalized Humbert group
$\langle b_{1},\ldots,b_{2s-1}\rangle=H_{0} \cong {\mathbb Z}_{2}^{2s-1}$ of type $2s-1$.

Let us consider the surjective homomorphism
$$\begin{array}{c}
\theta:H_{0} \to H=\langle a_{1},\ldots,a_{s-1}\rangle \cong {\mathbb Z}_{2}^{s-1}\\
b_{2k-1}, b_{2k} \mapsto a_{k}, \; k=1,\ldots, s-1,\\
b_{2s-1},b_{s} \mapsto a_{s}=a_{1}a_{2}\cdots a_{s-1}.
\end{array}
$$

The kernel of $\theta$ is given by 
$$K=\langle b_{1}b_{2}, b_{3}b_{4},\ldots,b_{2k-1}b_{2k}, \ldots, b_{2s-3}b_{2s-2},b_{1}b_{3}b_{5}\cdots b_{2s-1}\rangle \cong {\mathbb Z}_{2}^{s}$$
and, as the only non-trivial elements acting with fixed points are $b_{1},\ldots, b_{2s-1}$ and $b_{2s}=b_{1}b_{2}\cdots b_{2s-1}$, it follows that 
$K$ acts freely on $D$.  It follows that $R=D/K$ is a closed Riemann surface of genus (by the Riemann-Hurwitz formula) $g_{R}=1+2^{s-2}(s-2)$. 
In order to write equations for $R$, we need to compute a set of generators of ${\mathbb C}[z_{1},\ldots,z_{2s-1}]^{K}$, the algebra of $K$-invariant polynomials. Since the linear action of $K$ is given by diagonal matrices, a set of generators can be found to be 
$t_{1}=z_{1}^{2}, t_{2}=z_{2}^{2},\ldots, t_{2s-1}=z_{2s-1}^{2},$
together with the monomials of the form
$t_{\alpha}=(z_{1}z_{2})^{\alpha_{1}}(z_{3}z_{4})^{\alpha_{2}}\cdots(z_{2s-3}z_{2s-2})^{\alpha_{s-1}} z_{2s-1}^{\alpha_{s}},$
where $\alpha=(\alpha_{1},\ldots,\alpha_{s}) \in V_{s}$. As $V_{s}$ has cardinality $2^{s-1}-1$, the number of the above set of generators is $N=2^{s-1}+2s-2$.
Using the map $\Phi:D \to {\mathbb C}^{N}$, whose coordinates are $t_{1},\ldots,t_{2s-1}$ and the monomials $t_{\alpha}$, $\alpha \in V_{s}$, one obtains that the Riemann surface induced by $\Phi(D)$ is isomorphic to $R$ and that its equations are given by

$$
\Phi(D)=\left\{ \begin{array}{rrl}
t_{1}+t_{2}+t_{3}=0,& \lambda t_{1}+t_{2}+t_{4}=0,\\
\mu_{1,1}t_{1}+t_{2}+t_{5}=0,& \mu_{1,2}t_{1}+t_{2}+t_{6}=0,\\
\vdots &  \vdots  \\
\mu_{k,1}t_{1}+t_{2}+t_{2k+3}=0,& \mu_{k,2}t_{1}+t_{2}+t_{2k+4}=0,\\
\vdots &  \vdots \\
\mu_{s-2,1}t_{1}+t_{2}+t_{2s-1}=0,& \mu_{s-2,2}t_{1}+t_{2}+1=0,\\
t_{\alpha}^{2}=(t_{3}t_{4})^{\alpha_{2}}\cdots(t_{2s-3}t_{2s-2})^{\alpha_{s-1}} t_{2s-1}^{\alpha_{s}}; &\alpha=(\alpha_{1},\ldots,\alpha_{s}) \in V_{s} 
\end{array}
\right\}.
$$

The first linear equations permit us to write $t_{2},\ldots,t_{2s-1}$ in terms of $t_{1}$ as follows:
$$t_{2}=-1-\mu_{s-2,2}t_{1}, \; t_{3}=1+(\mu_{s-2,2}-1)t_{1}, \; t_{4}=1+(\mu_{s-2,2}-\lambda)t_{1}$$
$$t_{2k+3}=1+(\mu_{s-2,2}-\mu_{k,1})t_{1}, \quad t_{2k+4}=1+(\mu_{s-2,2}-\mu_{s,2})t_{1}, \quad k=1,\ldots, s-3,$$
$$t_{2s-1}=1+(\mu_{s-2,2}-\mu_{s-2,1})t_{1}.$$

We may then eliminate the variables $t_{2},\ldots,t_{2s-1}$ and just keep the variables $t_{1}$ and $t_{\alpha_{1},\ldots,\alpha_{s}}$. 
Let us set $t_{1}=z$ and $t_{\alpha}=w_{\alpha}$, for $\alpha \in V_{s}$.
In these new $2^{s-1}$ coordinates, the above curve is isomorphic to the one given by
$$
C=\left\{ \begin{array}{c}
w_{\alpha}^{2}=z^{\alpha_{1}}(-1-\mu_{s-2,2}z)^{\alpha_{1}}
(1+(\mu_{s-2,2}-1)z)^{\alpha_{2}}(1+(\mu_{s-2,2}-\lambda)z )^{\alpha_{2}}\cdots\\
\cdots  (1+(\mu_{s-2,2}-\mu_{k,1})z)^{\alpha_{k+2}}(1+(\mu_{s-2,2}-\mu_{k,2})z)^{\alpha_{k+2}} \cdots \\
\cdots (1+(\mu_{s-2,2}-\mu_{s-3,1})z)^{\alpha_{s-1}}(1+(\mu_{s-2,2}-\mu_{s-3,2})z)^{\alpha_{s-1}} (1+(\mu_{s-2,2}-\mu_{s-2,1})z)^{\alpha_{s}},\\ 
\alpha=(\alpha_{1},\ldots,\alpha_{s}) \in V_{s}.\\
\end{array}
\right\}
$$

By making the choices as described in the hypothesis of the theorem for $K_{\alpha}$ and the values of $\eta_{0}$, $\eta_{1}$, $\eta_{2}$, $\eta_{3}$ and $\eta_{k,j}$, then the above curve can be written in the desired algebraic form.
If $\Phi_{1}:D \to {\mathbb C}^{2^{s-1}}$ is the map whose coordinates are $z$ and $w_{\alpha}$, where $\alpha \in V_{s}$, then $\Phi_{1}(D)=C$. 
If $\alpha=(\alpha_{1},\ldots,\alpha_{s})$ and $j=1,\ldots,s-1$, then 
the induced automorphisms $a_{j}$ acts by multiplication by $-1$ at coordinates $w_{\alpha}$ if $\alpha_{j}=1$ and acts by the identity on the rest of coordinates. The map
$\pi:C \to \widehat{\mathbb C}: (z,\{w_{\alpha \in V_{s}}\}) \mapsto (1+\mu_{s-2,2}z)/z$
is a regular branched cover with $H$ as its deck group and its satisfies that $P=\pi \circ \Phi_{1}$. The branch locus of $\pi$ is the set
$\{\infty,0,1,\lambda, \mu_{1,1}, \mu_{1,2},\ldots,\mu_{s-2,1}, \mu_{s-2,2}\}.$
All the above permits to observe the following.

\begin{lemm}\label{lema01}
The only non-trivial elements of $H$ acting with fixed points are $a_{1},\ldots,a_{s-1}$ and $a_{s}=a_{1}a_{2} \cdots a_{s-1}$. Moreover, $\pi({\rm Fix}(a_{1}))=\{\infty,0\}$,  $\pi({\rm Fix}(a_{2}))=\{1,\lambda\}$, $\pi({\rm Fix}(a_{k}))=\{\mu_{k-2,1}, \mu_{k-2,2}\}$, for $k=3,\ldots, s-1$, and 
$\pi({\rm Fix}(a_{s}))=\{\mu_{s-2,1}, \mu_{s-2,2}\}.$
It can be seen that, for $j=1,\ldots,s$, ${\rm Fix}(a_{j})$ has cardinality $2^{s-1}$.
\end{lemm}

If $2 \leq k \leq s$ is even
and $1 \leq i_{1}<i_{2}< \cdots < i_{k} \leq s$, then we consider the subgroup
$$H_{i_{1},i_{2},\ldots,i_{k}}=\langle a_{i_{1}}a_{i_{2}}, a_{i_{1}}a_{i_{3}},\ldots,a_{i_{1}}a_{i_{k}},a_{j}; j\in\{1,\ldots,s\}-\{i_{1},\ldots,i_{k}\}\rangle \cong{\mathbb Z}_{2}^{s-2}. $$

If $k=2$, then we have $s(s-1)/2$ such subgroups. Between them are 
$H_{1,2}=\langle a_{1}a_{2},a_{3},\ldots,a_{s}\rangle$, 
$H_{2,3}=\langle a_{2}a_{3},a_{4},\ldots,a_{1}\rangle$,  
$H_{3,4}=\langle a_{3}a_{4},a_{5},\ldots,a_{2}\rangle, \ldots$,  
$H_{s-1,s}=\langle a_{s-1}a_{s},a_{1},\ldots,a_{s-2}\rangle$ and 
$H_{s,1}=\langle a_{s}a_{1},a_{2},\ldots,a_{s-1}\rangle$.
The quotient orbifold $C/H_{j,j+1}$ has underlying Riemann surface structure $E_{j} \cong E_{\lambda_{j}}$, for $j=1,\ldots,s-1$, and $C/H_{s,1}$ has underlying Riemann surface structure $E_{s} \cong E_{\lambda_{s}}$.

\begin{lemm} \label{lema02}
With the above notations, the following hold for the above defined subgroups.
\begin{enumerate}
\item Any two such subgroups $H_{i_{1},i_{2},\ldots,i_{k}}$ and $H_{j_{1},j_{2},\ldots,j_{l}}$ commute.
\item The quotient $C/H_{i_{1},i_{2},\ldots,i_{k}}$ has genus $k-1$ and its underlying Riemann surface is given by the (elliptic) hyperelliptic curve 
$$\nu^{2}=(\upsilon-\rho_{i_{1},1})(\upsilon-\rho_{i_{1},2})\cdots(\upsilon-\rho_{i_{k},1})(\upsilon-\rho_{i_{k},2}),$$
where
$$\rho_{i_{j},1}=\left\{\begin{array}{cl}
\infty, & i_{j}=1\\
1, & i_{j}=2\\
\mu_{r-2,1}, & i_{j}=r \geq 3
\end{array}
\right. \quad 
\rho_{i_{j},2}=\left\{\begin{array}{cl}
0, & i_{j}=1\\
\lambda, & i_{j}=2\\
\mu_{r-2,2}, & i_{j}=r\geq 3
\end{array}
\right.
$$
In the case that $\rho_{i_{j},1}=\infty$, then the factor $(u-\rho_{i_{j},1})$ is deleted from the above expression.

\item The group generated by any two different such subgroups is $H$.

\end{enumerate}

\end{lemm}
\begin{proof}
Property (1) holds trivially as $H$ is an abelian group. Property (2) follows from Riemann-Hurwitz formula and Lemma \ref{lema01}. Property (3) is clear as in the product we obtain all the generators.
\end{proof}

The next result states that the sum of the genera appearing in all quotients of the form $C/H_{i_{1},i_{2},\ldots,i_{k}}$ is equal to the genus of $C$ (that is, the genus of $R$). Recall that we are considering $k$ even and $2 \leq k \leq s$.

\begin{lemm}\label{lema03}
$\sum_{k=2}^{s} \binom{r}{k} (k-1) \left(\frac{1+(-1)^{k}}{2}\right)
=1+2^{s-2}(s-2).$
\end{lemm}
\begin{proof}
If $f(x)=\frac{(1+x)^{s}}{2x}=\frac{1}{2}\sum_{k=0}^{s} \binom{s}{k} x^{k-1}$, then 
$f'(x)=\frac{(1+x)^{s-1}((s-1)x-1)}{2x^{2}}=\frac{1}{2}\sum_{k=0}^{s} \binom{s}{k} (k-1) x^{k-2}$,
and 
$2^{s-2}(s-2)=f'(1)+f'(-1)=\frac{1}{2}\sum_{k=0}^{s} \binom{s}{k} (k-1) (1+(-1)^{k}) =-1+\sum_{k=2}^{s} \binom{s}{k} \frac{(1+(-1)^{k})}{2} (k-1)$.
\end{proof}

We may apply Theorem \ref{coroKR} for $C$ using all the subgroups $H_{i_{1},\ldots,i_{k}}$ in order to obtain that $JC$ (so $JR$) is isogenous to the product of the jacobian varieties of all Riemann surfaces $C/H_{i_{1},\ldots,i_{k}}$ as desired. The equations of these curves are provided in Lemma \ref{lema02}.

\begin{rema}[A fiber product description]
Let us consider the $s$ elliptic curves $E_{j}=E_{j}(P)$, where $j=1,\ldots s$.
If we consider the degree two maps $\pi_{j}:E_{j} \to \widehat{\mathbb C}$ defined as $\pi_{j}(x,y)=x$, then we may perform the fiber product of the $s$ pairs $(E_{1},\pi_{1}),\ldots,(E_{s},\pi_{s})$. Such a fiber product is given by an affine curve 
$Z \subset {\mathbb C}^{s+1}$, formed of the tuples $(x,y_{1},\ldots,y_{s})$ so that $(x,y_{j}) \in E_{j}$, for $j=1,\ldots,s$. The curve $Z$ is reducible and contains two irreducible components, both of them being isomorphic, and it admits the group of automorphisms
$N=\langle f_{1},\ldots,f_{s}\rangle \cong {\mathbb Z}_{2}^{s}$, where
$f_{j}(x,y_{1},\ldots,y_{s})=(x,y_{1},\ldots,y_{j-1},-y_{j},y_{j+1},\ldots,y_{s})$. The two irreducible factors are permuted by some elements of $N$ and each one is invariant under a subgroup isomorphic to ${\mathbb Z}_{2}^{s-1}$. It can be shown that the Riemann surface defined by any of these two irreducible components is isomorphic to $X$ as defined in Theorem \ref{construccion}.
\end{rema}


\subsection{Proof of Theorem \ref{construccion2}}

Let $(\lambda_{1},\ldots,\lambda_{r}) \in \Delta_{r}$, $E_{\lambda_{j}}: y_{j}^{2}=x_{j}(x_{j}-1)(x_{j}-\lambda_{j})$ and $\pi_{j}:E_{\lambda_{j}} \to \widehat{\mathbb C}$ defined by $\pi_{j}(x_{j},y_{j})=x_{j}$. The locus of branch values of $\pi_{j}$ is the set $\{\infty,0,1,\lambda_{j}\}$.
Let us consider the generalized Humbert curve $F=C(\lambda_{1},\ldots,\lambda_{r}) \subset {\mathbb P}^{r+2}$ of type $r+2$, which is of genus 
$g=1+2^{r}(r-1)$, with corresponding generalized Humbert group $\langle b_{1},\ldots,b_{r+2}\rangle=H_{0} \cong {\mathbb Z}_{2}^{r+2}$ of type $r+2$.
As the only non-trivial elements of $H_{0}$ acting with fixed points are $b_{1},\ldots, b_{r+2}$ and $b_{r+3}=b_{1}b_{2}\cdots b_{r+2}$,  
the subgroup $K^{*}=\langle b_{1}b_{2}, b_{2}b_{3}\rangle \cong {\mathbb Z}_{2}^{2}$ acts freely on $F$. 
Then $Y=F/K^{*}$ is a closed Riemann surface of genus $g_{Y}=1+2^{r-2}(r-1)$. The quotient group $L=H_{0}/\langle b_{1}b_{2}, b_{2}b_{3}\rangle$ induces the group of automorphisms of $Y$. It can be seen that $L=\langle c_{1},\ldots,c_{r}\rangle \cong {\mathbb Z}_{2}^{r}$, where 
$$c_{j}(x,y_{1},\ldots,y_{r})=(x,y_{1},\ldots,y_{j-1}, - y_{j},y_{j+1},\ldots,y_{r}), \; j=1,...,r.$$ 

The map $\pi:Y \to \widehat{\mathbb C}: (x,y_{1},\ldots,y_{r}) \mapsto x$, 
is a regular branched cover, with $H$ as its deck group, and its branch locus is the set
$\{\infty,0,1,\lambda_{1},\ldots,\lambda_{r}\}.$

\begin{rema}
With the above description, we obtain a set of equations for $Y$ as 
$$
\left\{\begin{array}{c}
w_{1}^{2}=(\lambda_{r}-\lambda_{1})u+1, \; 
w_{2}^{2}=(\lambda_{r}-\lambda_{2})u+1, \cdots, 
w_{r-1}^{2}=(\lambda_{r}-\lambda_{r-1})u+1,\\
w_{r}^{2}=-u(\lambda_{r}u+1)((\lambda_{r}-1)u+1)
\end{array}
\right\}.
$$

The equations for $E_{\lambda_{j}}$ can also be written as (for $j=1,\ldots,r-1$)
$$
\left\{\begin{array}{c}
w_{1}^{2}=(\lambda_{r}-\lambda_{1})u+1, \cdots, 
w_{j-1}^{2}=(\lambda_{r}-\lambda_{j-1})u+1\\
w_{j}^{2}=(\lambda_{r}-\lambda_{j})u+1, \cdots, 
w_{r-1}^{2}=(\lambda_{r}-\lambda_{r-1})u+1\\
w_{r}^{2}=-u(\lambda_{r}u+1)((\lambda_{r}-1)u+1)((\lambda_{r}-\lambda_{j})u+1)
\end{array}
\right\}
$$
and for $E_{\lambda_{r}}$ as 
$$
\left\{\begin{array}{c}
w_{1}^{2}=(\lambda_{r}-\lambda_{1})u+1, \cdots, 
w_{r-1}^{2}=(\lambda_{r}-\lambda_{r-1})u+1\\
w_{r}^{2}=-u(\lambda_{r}u+1)((\lambda_{r}-1)u+1)
\end{array}
\right\}.
$$
\end{rema}

\begin{lemm}\label{lema1}
The only non-trivial elements of $L$ acting with fixed points are $c_{1},\ldots,c_{r}$ and $c_{r+1}=c_{1}c_{2} \cdots c_{r}$. Moreover, $\pi({\rm Fix}(c_{j}))=\lambda_{j}$, for 
$j=1,\ldots,r$, and $\pi({\rm Fix}(c_{r+1}))=\{\infty,0,1\}.$

\end{lemm}
\begin{proof}
A non-trivial element of $L$ has the form
$c(x,y_{1},\ldots,y_{r})=(x,(-1)^{\alpha_{1}}y_{1},\ldots,(-1)^{\alpha_{r}}y_{r}),$
where $\alpha_{1},\ldots,\alpha_{r} \in \{0,1\}$ and $\alpha_{1}+\cdots+\alpha_{r}>0$.
A point $(x,y_{1},\ldots,y_{r}) \in C$ is a fixed point of $c$ if and only if $y_{j}=0$ for $\alpha_{j}=1$. The equality $y_{j}=0$ is equivalent to havimg $x \in \{\infty,0,1,\lambda_{j}\}$. The values $x \in \{\infty,0,1\}$ produce fixed points for $c_{r+1}$. Also, as we are assume the values $\lambda_{j}$ to be different, it follows that the only possibility is to have only one $j$ with $\alpha_{j}=1$.
\end{proof}

It can be seen that, for $j=1,\ldots,r$, ${\rm Fix}(c_{j})$ has cardinality $2^{r-1}$ and that ${\rm Fix}(c_{1}\cdots c_{r})$ has cardinality $3 \times 2^{r-1}$. In particular, for $r=3$, the surface $Y$ is hyperelliptic with $c_{1}c_{2}c_{3}$ as its hyperelliptic involution.

If either $1 \leq k \leq r$ is odd or $4 \leq k \leq r$ is even,
and $\{i_{1},\ldots,i_{k}\} \subset \{1,\ldots,r\}$ with $1 \leq i_{1}<i_{2}< \cdots < i_{k} \leq r$, then we consider the subgroup
$$L_{i_{1},i_{2},\ldots,i_{k}}=\langle c_{i_{1}}c_{i_{2}}, c_{i_{1}}c_{i_{3}},\ldots,c_{i_{1}}c_{i_{k}},c_{j}; j\in\{1,\ldots,r\}-\{i_{1},\ldots,i_{k}\}\rangle \cong{\mathbb Z}_{2}^{r-1}. $$

Note that for $k=1$ we have the subgroups 
$L_{j}=\langle c_{1},\ldots,c_{j-1},c_{j+1},\ldots,c_{r}\rangle \cong {\mathbb Z}_{2}^{r-1}.$
If $$Q_{j}:Y \to E_{\lambda_{j}}: (x,y_{1},\ldots,y_{r}) \mapsto (x,y_{j}),$$
then $Q_{j}$ is a regular branched cover with deck group being $L_{j}$. The branch locus of $Q_{j}$ is the set
$$\{(\lambda_{j},y_{i}): y_{i}^{2}=\lambda_{i}(\lambda_{i}-1)(\lambda_{i}-\lambda_{j}), \; i=1,\ldots,r, i \neq j\}.$$

\begin{lemm} \label{lema2}
With the above notations, the following hold for the above defined subgroups.
\begin{enumerate}
\item Any two such subgroups commute.
\item The quotient $Y/L_{i_{1},i_{2},\ldots,i_{k}}$ is an orbifold of genus $(k+1)/2$ if $k$ is odd and genus $(k-2)/2$ if $k$ is even. Moreover, the underlying Riemann surface is given by the hyperelliptic curve 
$$w^{2}=\left\{\begin{array}{ll}
z(z-1)(z-\lambda_{i_{1}})(z-\lambda_{i_{2}})\cdots(z-\lambda_{i_{k}}), &  \mbox{ if $k$ is odd},\\
(z-\lambda_{i_{1}})(z-\lambda_{i_{2}})\cdots(z-\lambda_{i_{k}}), & \mbox{ if $k$ is even}.
\end{array}
\right.
$$

\item The group generated by any two different such subgroups is $L$.

\end{enumerate}

\end{lemm}
\begin{proof}
Property (1) holds trivially as $L$ is an abelian group. Property (2) follows from Riemann-Hurwitz formula and Lemma \ref{lema1}. Property (3) is clear.
\end{proof}

The next result states that the sum of the genera appearing in all quotients of the form $Y/L_{i_{1},i_{2},\ldots,i_{k}}$ is equal to the genus of $Y$.

\begin{lemm}\label{lema3}
$\sum_{k=1}^{r} \binom{r}{k} \frac{(1-(-1)^k)}{2} \frac{(k+1)}{2} + 
\sum_{k=4}^{r} \binom{r}{k} \frac{(1+(-1)^k)}{2} \frac{(k-2)}{2}
=1+2^{r-2}(r-1).$
\end{lemm}
\begin{proof}
If  
$f_{1}(x)=x\frac{(1+x)^{r}}{4}=\frac{1}{4}\sum_{k=0}^{r} \binom{r}{k} x^{k+1}$ and 
$f_{2}(x)=\frac{(1+x)^{r}}{4x^{2}}= \frac{1}{4}\sum_{k=0}^{r} \binom{r}{k} x^{k-2}$, then 
$f'_{1}(x)=\frac{(1+x)^{r}(1+(1+r)x)}{4}=\frac{1}{4}\sum_{k=0}^{r} \binom{r}{k} (k+1) x^{k}$ and 
$f'_{2}(x)=\frac{(1+x)^{r-1}(r-2x(1+x))}{4x^{2}}=\frac{1}{4}\sum_{k=0}^{r} \binom{r}{k} (k-2)x^{k-3}$. Then
$2^{r-3}(r+2)=f'_{1}(1)-f'_{1}(-1)=\frac{1}{4}\sum_{k=0}^{r} \binom{r}{k} (k+1)(1-(-1)^{k})=\sum_{k=1}^{r} \binom{r}{k} \frac{(1-(-1)^{k})}{2} \frac{(k+1)}{2}$ and 
$2^{r-3}(r-4)=f'_{2}(1)+f'_{2}(-1)=\frac{1}{4}\sum_{k=0}^{r} \binom{r}{k} (k-2)(1+(-1)^{k})=-1+\sum_{k=4}^{r} \binom{r}{k} \frac{(1+(-1)^k)}{2} \frac{(k-2)}{2}$.
By adding these two equalities we obtain the desired result.
\end{proof}

We may now apply Theorem \ref{coroKR} for $Y$ using the subgroups $L_{i_{1},\ldots,i_{k}}$ in order to obtain that
$JY$ is isogenous to a product of the form $E_{\lambda_{1}} \times \cdots \times E_{\lambda_{r}} \times A_{g-r}$, where  $A_{g-r}$ is the product of the jacobian varieties of all elliptic/hyperelliptic Riemann surfaces $Y/L_{i_{1},\ldots,i_{k}}$, for $k \geq 2$.

\begin{rema}[A fiber product description of $Y$]
We may generalize the fiber product description done for $r=2$ in Remark \ref{observa2}, for $r \geq 2$, as follows.
Let $\pi_{j}:E_{\lambda_{j}} \to \widehat{\mathbb C}$ be defined by $\pi_{j}(x_{j},y_{j})=x_{j}$. The locus of branch values of $\pi_{j}$ is the set $\{\infty,0,1,\lambda_{j}\}$.
An affine model of the fiber product of the $r$ pairs $(E_{\lambda_{1}},\pi_{1}),\ldots,(E_{\lambda_{r}},\pi_{r})$ is 
$$\widehat{Y}=\left\{(x,y_{1},\ldots,y_{r}): y_{j}^{2}=x(x-1)(x-\lambda_{j}), \; j=1,\ldots,r \right\},$$
which is irreducible and has singular points at those points with first coordinate $x\in \{\infty,0,1\}$. The Riemann surface $T$ defined by $\widehat{Y}$ (after desingularization) admits a group of conformal automorphisms $H \cong {\mathbb Z}_{2}^{r}$ so that $T/H$ is the Riemann sphere with conical points (each one of order two) at $\infty$, $0$, $1$, $\lambda_{1},\ldots, \lambda_{r}$. In particular, it has genus $g=1+2^{r-2}(r-1)$. It can be checked that $T$ is isomorphic to the surface $Y$ of Theorem \ref{construccion2}.
\end{rema}

\section{Explicit examples}

\subsection{Example of genus five} \label{g=5}
In Theorem \ref{construccion2}, assume $r=3$.
In this case, $Y$ has genus five and $JY$ is isogenous to $E_{\lambda_{1}} \times E_{\lambda_{2}} \times E_{\lambda_{3}} \times JT$, where 
$E_{\lambda_{1}}: y^{2}=x(x-1)(x-\lambda_{1})$,
$E_{\lambda_{2}}: y^{2}=x(x-1)(x-\lambda_{2})$, 
$E_{\lambda_{3}}: y^{2}=x(x-1)(x-\lambda_{3})$ and 
$T: y^{2}=x(x-1)(x-\lambda_{1})(x-\lambda_{2})(x-\lambda_{3})$.
The corresponding subgroups of $H$ are in this case
$L_{1}=\langle c_{2},c_{3}\rangle$, 
$L_{2}=\langle c_{1},c_{3}\rangle$, 
$L_{3}=\langle c_{1},c_{2}\rangle$ and 
$L_{1,2,3}=\langle c_{1}c_{2}, c_{1}c_{3}\rangle.$

\begin{coro}
If $\lambda_{3}=\lambda_{1}/\lambda_{2}$, then $JY$ isogenous to the product of $5$ elliptic curves. So, it provides a $2$-dimensional family of curves $Y$ of genus five with $JY$ isogenous to the product of five elliptic curves.
\end{coro}
\begin{proof}
If $\lambda_{3}=\lambda_{1}/\lambda_{2}$, then $T$ admits the involution $(x,y) \mapsto (\lambda_{1}/x, \lambda_{1}^{3/2} y/x^{3})$,  with exactly two fixed points. It follows that $JT$ is isogenous to the product of two elliptic curves.
\end{proof}

\subsection{Example of genus nine}\label{coro3}
Next, we provide a two-dimensional family of genus nine Riemann surfaces whose jacobian varieties are isogenous to the product of nine elliptic curves.
For $(\lambda, \mu) \in \Delta_{2}$ set 
$\mu_{1,1}=\mu, \; \mu_{1,2}=\lambda/\mu$, $\mu_{2,1}=\lambda(\mu-1)/(\mu-\lambda)$, $\mu_{2,2}=(\mu-\lambda)/(\mu-1)$,
$$
\begin{array}{ll}
K_{1}=(\mu_{2,2}-\mu_{1,1})(\mu_{2,2}-\mu_{1,2})(\mu_{2,2}-\mu_{2,1}),&
K_{2}=(\mu_{2,2}-1)(\mu_{2,2}-\lambda)(\mu_{2,2}-\mu_{2,1}),\\
K_{3}=(\mu_{2,2}-1)(\mu_{2,2}-\lambda)(\mu_{2,2}-\mu_{1,1})(\mu_{2,2}-\mu_{1,2}), &
K_{4}=-\mu_{2,2}(\mu_{2,2}-\mu_{2,1}),\\
K_{5}=-\mu_{2,2}(\mu_{2,2}-\mu_{1,1})(\mu_{2,2}-\mu_{1,2}), &
K_{6}=-\mu_{2,2}(\mu_{2,2}-1)(\mu_{2,2}-\lambda).
\end{array}
$$

Let $R$ be the genus nine Riemann surface defined by the curve
$$\left\{\begin{array}{lcl}
w_{1}^{2}&=&K_{1}
\left(z-\dfrac{1}{\mu_{1,1}-\mu_{2,2}} \right)
\left(z-\dfrac{1}{\mu_{1,2}-\mu_{2,2}} \right)
\left(z-\dfrac{1}{\mu_{2,1}-\mu_{2,2}} \right)  \\
w_{2}^{2}&=&K_{2}
\left(z-\dfrac{1}{1-\mu_{2,2}} \right)
\left(z-\dfrac{1}{\lambda-\mu_{2,2}} \right)
\left(z-\dfrac{1}{\mu_{2,1}-\mu_{2,2}} \right)  \\
w_{3}^{2}&=&K_{3}
\left(z-\dfrac{1}{1-\mu_{2,2}} \right)
\left(z-\dfrac{1}{\lambda-\mu_{2,2}} \right)
\left(z-\dfrac{1}{\mu_{1,1}-\mu_{2,2}} \right)
\left(z-\dfrac{1}{\mu_{1,2}-\mu_{2,2}} \right)  \\
w_{4}^{2}&=&K_{4}
z\left(z+\dfrac{1}{\mu_{2,2}} \right)
\left(z-\dfrac{1}{\mu_{2,1}-\mu_{2,2}} \right)  \\
w_{5}^{2}&=&K_{5}
z\left(z+\dfrac{1}{\mu_{2,2}} \right)
\left(z-\dfrac{1}{\mu_{1,1}-\mu_{2,2}} \right)
\left(z-\dfrac{1}{\mu_{1,2}-\mu_{2,2}} \right)  \\
w_{6}^{2}&=&K_{6}
z\left(z+\dfrac{1}{\mu_{2,2}} \right)
\left(z-\dfrac{1}{1-(\mu_{2,2}} \right)
\left(z-\dfrac{1}{\lambda-\mu_{2,2}} \right)  \\
w_{7}^{2}&=&w_{3}^{2}w_{4}^{2}
\end{array}
\right\}
$$
By Theorem \ref{construccion} (and Lemma \ref{lema02}), $JR \cong_{isog.} E_{1} \times \cdots \times E_{6} \times JT$,  where  
$E_{1}:\; y^{2}=x(x-1)(x-\lambda)$, $E_{2}:\; y^{2}=(x-1)(x-\lambda)(x-\mu_{1,1})(x-\mu_{1,2})$,
$E_{3}:\; y^{2}=(x-\mu_{1,1})(x-\mu_{1,2})(x-\mu_{2,1})(x-\mu_{2,2})$, $E_{4}:\; y^{2}=x(x-\mu_{2,1})(x-\mu_{2,2})$,
$E_{5}:\; y^{2}=x(x-\mu_{1,1})(x-\mu_{1,2})$, $E_{6}:\; y^{2}=(x-1)(x-\lambda)(x-\mu_{2,1})(x-\mu_{2,2})$
and
$T: y^{2}=x(x-1)(x-\lambda)(x-\mu_{1,1})(x-\mu_{1,2})(x-\mu_{2,1})(x-\mu_{2,2}).$
The group $J=\langle f_{1}(x)=\lambda/x, f_{2}(x)=\lambda (x-1)/(x-\lambda) \rangle \cong {\mathbb Z}_{2}^{2}$ keeps invariant the set 
$\{\infty, 0, 1, \lambda, \mu_{1,1}, \mu_{1,2}, \mu_{2,1}, \mu_{2,2}\}$. In this way, $T$ admits the following automorphisms
$$A_{1}(x,y)=\left( \lambda / x, \lambda^{2} y / x^{4}\right), \;
A_{2}(x,y)=\left( \lambda(x-1)/(x-\lambda), \lambda^{2}(\lambda-1)^{2} y/(x-\lambda)^{4}\right).$$

We may see that $\langle A_{1}, A_{2}\rangle \cong {\mathbb Z}_{2}^{2}$ and that each of the involutions $A_{1}$, $A_{2}$ and $A_{1}\circ A_{2}$ acts with exactly $4$ fixed points on $T$. The quotients $T/\langle A_{1} \rangle$, $T/\langle A_{2} \rangle$ and $T/\langle A_{1} \circ A_{2}\rangle$ have genus one. We may apply Theorem \ref{coroKR} to $T$ using the three cyclic groups of order two in order to see that $JT$ is isogenous to the product of three elliptic curves. In this way, 
$JR$ is isogenous to the product of nine elliptic curves.

\subsection{Example of genus thirteen}\label{g=13}
Set, in Theorem \ref{construccion2}, the case $r=4$.
In this case, $Y$ has genus $13$ and $JY$ is isogenous to  $E_{\lambda_{1}} \times E_{\lambda_{2}} \times E_{\lambda_{3}} \times E_{\lambda_{4}} \times JT_{1} \times JT_{2} \times JT_{3} \times JT_{4} \times E_{5}$, where 
$E_{\lambda_{1}}: y^{2}=x(x-1)(x-\lambda_{1})$,
$E_{\lambda_{2}}: y^{2}=x(x-1)(x-\lambda_{2})$, 
$E_{\lambda_{3}}: y^{2}=x(x-1)(x-\lambda_{3})$,
$E_{\lambda_{4}}: y^{2}=x(x-1)(x-\lambda_{4})$,   
$T_{1}: y^{2}=x(x-1)(x-\lambda_{1})(x-\lambda_{2})(x-\lambda_{3}),\; T_{2}: y^{2}=x(x-1)(x-\lambda_{1})(x-\lambda_{2})(x-\lambda_{4})$, 
$T_{3}: y^{2}=x(x-1)(x-\lambda_{1})(x-\lambda_{3})(x-\lambda_{4}),\; T_{4}: y^{2}=x(x-1)(x-\lambda_{2})(x-\lambda_{3})(x-\lambda_{4})$ and
$E_{5}: y^{2}=(x-\lambda_{1})(x-\lambda_{2})(x-\lambda_{3})(x-\lambda_{4})$.
The corresponding subgroups of $H$ are in this case
$L_{1}=\langle c_{2},c_{3},c_{4}\rangle$,
$L_{2}=\langle c_{1},c_{3},c_{4}\rangle$, 
$L_{3}=\langle c_{1},c_{2},c_{4}\rangle$, 
$L_{4}=\langle c_{1},c_{2},c_{3}\rangle$,
$L_{1,2,3}=\langle c_{1}c_{2}, c_{1}c_{3}, c_{4}\rangle$, 
$L_{1,2,4}=\langle c_{1}c_{2}, c_{1}c_{4}, c_{3}\rangle$, 
$L_{1,3,4}=\langle c_{1}c_{3}, c_{1}c_{4}, c_{2}\rangle$,
$L_{2,3,4}=\langle c_{2}c_{3}, c_{2}c_{4}, c_{1}\rangle$ and  
$L_{1,2,3,4}=\langle c_{1}c_{2},c_{1}c_{3},c_{1}c_{4}\rangle$.

\begin{coro}
If $\lambda_{3}=\lambda_{1}/\lambda_{2}$, $\lambda_{4}=\lambda_{1}(\lambda_{2}-1)/(\lambda_{2}-\lambda_{1})$ and $\lambda_{2}^{2}(1+\lambda_{1})-4\lambda_{1}\lambda_{2}+\lambda_{1}(1+\lambda_{1})=0$, then $JY$ isogenous to the product of $13$ elliptic curves. 
\end{coro}
\begin{proof}
If $a_{1}(z)=\lambda_{1}/z$ and $a_{2}(z)=\lambda_{1}(z-1)/(z-\lambda_{1})$, then the group generated by them is isomorphic to ${\mathbb Z}_{2}^{2}$. Since $a_{1}$ permutes in pairs the elements in $\{\infty,0,1,\lambda_{1},\lambda_{2},\lambda_{3}\}$, it follows that $JT_{1}$ is isogenous to the product of two elliptic curves. Similarly, as $a_{2}$ permutes in pairs the elements in $\{\infty,0,1,\lambda_{1},\lambda_{2},\lambda_{4}\}$, it follows that $JT_{2}$ is isogenous to the product of two elliptic curves and as 
$a_{2}a_{1}$ permutes in pairs the elements in $\{\infty,0,1,\lambda_{1},\lambda_{3},\lambda_{4}\}$, it follows that $JT_{3}$ is isogenous to the product of two elliptic curves. In this way, under the above assumptions, $JY$ is isogenous to the product of $11$ elliptic curves and $JT_{4}$. If we also assume that $\lambda_{2}^{2}(1+\lambda_{1})-4\lambda_{1}\lambda_{2}+\lambda_{1}(1+\lambda_{1})=0$, then $a_{3}(z)=\lambda_{2}(z-\lambda_{3})/(z-\lambda_{2})$ permutes in pairs the elements of the set $\{\infty,0,1,\lambda_{2},\lambda_{3},\lambda_{4}\}$. In this case, $JT_{4}$ is also isogenous to the product of two elliptic curves\end{proof}

\begin{rema}
Examples of values as above are $\lambda_{1}=2$ and $\lambda_{2}=(4+i\sqrt{2})/3$; so $\lambda_{3}=(4-i\sqrt{2})/3$ and $\lambda_{4}=-i\sqrt{2}$.
\end{rema}

\bibliographystyle{amsplain}

\end{document}